\documentclass[sn-mathphys]{sn-jnl}

\jyear{2021}%

\theoremstyle{thmstyleone}%
\newtheorem{theorem}{Theorem}
\newtheorem{proposition}[theorem]{Proposition}%
\newtheorem{corollary}[theorem]{Corollary}
\newtheorem{lemma}[theorem]{Lemma}

\theoremstyle{thmstyletwo}%
\newtheorem{example}{Example}%
\newtheorem{remark}{Remark}%

\theoremstyle{thmstylethree}%
\newtheorem{definition}{Definition}%

\begin{document}

\title[A new generalized inverse in Banach algebras]{A new extension of generalized Drazin inverse in Banach algebras}


\author[1]{\fnm{Yanxun} \sur{Ren}}\email{renyanxun110@126.com}

\author*[2]{\fnm{Lining} \sur{Jiang}}\email{jianglining@bit.edu.cn}
\equalcont{These authors contributed equally to this work.}


\affil*[1]{\orgdiv{School of Mathematics and Statistics}, \orgname{Beijing Institute of Technology}, \orgaddress{\city{Beijing}, \postcode{100081},  \country{China}}}

\affil[2]{\orgdiv{School of Mathematics and Statistics}, \orgname{Beijing Institute of Technology}, \orgaddress{\city{Beijing}, \postcode{100081},  \country{China}}}



\abstract{In this paper, we introduce and study a new generalized inverse, called ag-Drazin inverses in a Banach algebra $\mathcal{A}$ with unit $1$. An element $a\in\mathcal{A}$ is ag-Drazin invertible if there exists $x\in\mathcal{A}$ such that $ax=xa, \, xax=x \ {\rm and} \ a-axa\in\mathcal{A}^{acc}$, where $\mathcal{A}^{acc}\triangleq\{a\in\mathcal{A}: a-\lambda 1 \ {\rm is \ generalized \ Drazin\ invertible}  \ {\rm for \ all} \ \lambda\in\mathbb{C}\backslash\{0\}\}.$ Using idempotent elements, we characterize this inverse and give some its representations. Also, we prove that $a\in\mathcal{A}$ is ag-Drazin invertible if and only if $0$ is not an accumulation point of $\sigma_{d}(a)$, where $\sigma_{d}(a)$ is the generalized Drazin spectrum of $a$. }




\maketitle

\section{Introduction}\label{sec1}

Throughout this paper, $\mathcal{A}$ stands for a Banach algebra over complex filed $\mathbb{C}$ with idenity $1$. We use ${\rm iso\,}M$ and ${\rm acc\,}M$ to denote the set of all isolated points and  accumulation points of $M\subseteq\mathbb{C}$, respectively. For $a\in\mathcal{A}$, by $\{a\}^{'}$ we denote $\{b\in R: ab=ba\}$ and accordingly, by $\{a\}^{''}$ we denote $\{\{a\}^{'}\}^{'}$. Let $\sigma(a)$ be the spectrum of $a$. Recall that $a$ is a quasinilpotent element if $a\in\mathcal{A}^{qnil}$, where $\mathcal{A}^{qnil}=\{a\in\mathcal{A}: \sigma(a)=\{0\}\}$.  Given $\epsilon>0$ and $\lambda\in\mathbb{C}$, $B(\lambda, \epsilon)=\{\mu\in\mathbb{C}: \mid\mu-\lambda\mid<\epsilon\}$, and $B(\lambda, \epsilon)^{\circ}=\{\mu\in\mathbb{C}: 0<\mid\mu-\lambda\mid<\epsilon\}$.

Recall that an element $a\in R$ is Drazin invertible \cite{MPD} if there exists an element $x\in\{a\}^{'}$ such that
\begin{equation*}
xax=x \ {\rm and} \ a^{k}xa=a^{k} \ {\rm for \ some \ integer} \ k\geq0.
\end{equation*}
The preceding $x$ is unique if it exists, and denote it by $a^{D}$.  The Drazin index of $a$, denoted by $\hbox{ind}(a)$, is the smallest non-negative integer $k$ satisfying $a^{k}xa=a^{k}$. In particular, an element $a$ is called group invertible if ${\rm ind}(a)=1$, and the group inverse of $a$ is denoted by $a^{\sharp}$. By $\mathcal{A}^{\sharp}$ we denote the set of all group invertible elements in $\mathcal{A}$.
The concept of generalized Drazin inverse in Banach algebras was introduced by Koliha \cite{Kol1996}. An element $a\in\mathcal{A}$ is generalized Drazin invertible if there exists $x\in\{a\}^{'}$ such that 
\begin{equation*}
 xax=x \ {\rm and} \ a-a^{2}x\in\mathcal{A}^{qnil}.
\end{equation*}
If such $x$ exists, then it is unique, and denote it by $a^{d}$. We use $\mathcal{A}^{d}$ to denote the set of all generalized Drazin invertible elements of $\mathcal{A}$. The generalized Drazin spectrum of $a$ is defined by $\sigma_{d}(a)=\{\lambda\in\mathbb{C}: a-\lambda 1 \notin\mathcal{A}^{d}\}$. Accordingly, $\rho_{d}(a)=\mathbb{C}\backslash\sigma_{d}(a)$.  In particularly, $\sigma_{d}(a)$ may be a empty set. For example, when $a$ is a quasinilpotent element, one has ${\rm acc\,}\sigma(a)=\emptyset$, and hence $\sigma_{d}(a)=\emptyset$. It follows from \cite[Theorem 4.2]{Kol1996} that $0\notin\sigma_{d}(a)$ if and only if $0\notin{\rm acc\,}\sigma(a)$. More interesting properties of the g-Drazin inverse
can be found in a recent papers \cite{CVE2011,Mos2017,Mos20173}.

In \cite{Mos2020}, Mosi\'c introduced a new generalized inverse, which is an extension of generalized Drazin inverse in a Banach algebra. We call $a\in\mathcal{A}$ extended generalized Drazin invertible (in short, eg-Drazin invertible) if there exists $x\in\{a\}^{'}$ such that 
\begin{equation*}
 xax=x \ {\rm and} \ a-a^{2}x\in\mathcal{A}^{d}.
\end{equation*}
Such $x$ is called a eg-Drazin inverse of $a$, and denote it by $a^{ed}$. $\mathcal{A}^{ed}$ denotes the set of all eg-Drazin invertible elements of $\mathcal{A}$. It was proved in \cite[Theorem 2]{Mos2020} that $a$ is eg-Drazin invertible if and only if $a$ is generalized Drazin invertible if and only if $0\notin{\rm acc\,}\sigma(a)$.

Let $L(\mathcal{X})$ denote the algebra of all bounded linear operators on a complex Banach space $\mathcal{X}$. For $T\in L(\mathcal{X})$, by $n(T)$ and $d(T)$ we denote the dimension of null space $N(T)$  and the codimension of range space $R(T)$, respectively. We call $T$ a Fredholm operator if both $n(T)$ and $d(T)$ are finite, and $T$ a Browder operator if $T$ is a Fredholm operator and $0\in{\rm iso\,}\sigma(T)$. The Fredholm spectrum $\sigma_{e}(T)$ and the Browder spectrum $\sigma_{b}(T)$ are the set of all $\lambda\in\mathbb{C}$ such that $T-\lambda I$ is not a Fredholm operator and Browder operator, respectively. Recall that $T$ is a Riesz operator if $\sigma_{e}(T)=\{0\}$, i.e. $T-\lambda I$ is Browder for $\lambda\in\mathbb{C}\backslash\{0\}$ (see \cite[Theorem 3.2]{Aie2018}). Obviously, if $T$ is Riesz, then $\sigma_{d}(T)\subseteq\{0\}$.

In \cite{Ziv2017}, \v Zivkovi\'c-Zlatanovi\'c introduced a concept of generalized Drazin-Riesz invertible operators in $L(\mathcal{X})$.
An operator $T\in L(\mathcal{X})$ is generalized Drazin-Riesz invertible if there exists $S\in\{T\}^{'}$ such that 
\begin{equation*}
STS=S \ {\rm and} \ T-T^{2}S \ {\rm is \ a \ Riesz \ operator}.
\end{equation*}
 In \cite{Aba2021}, Abad and Zguitti further investigated generalized Drazin-Riesz invertible operators, and proved that $T$ is a generalized Drazin-Riesz invertible operators if and only if $0\notin{\rm acc\,}\sigma_{b}(T)$ in Theorem 2.9 of \cite{Aba2021}.

Let $\mathcal{A}^{acc}\triangleq\{a\in\mathcal{A}: a-\lambda 1\in\mathcal{A}^{d} \ {\rm for \ all} \ \lambda\in\mathbb{C}\backslash\{0\}\}.$
Motivated by generalization of the Drazin inverse to the generalized Drazin inverse, to further
extend the notion of generalized Drazin inverse, we use $\mathcal{A}^{acc}$ rather that $\mathcal{A}^{qnil}$ in the definition of generalized Drazin inverse. 
\begin{definition}
We say that $a\in\mathcal{A}$ is $\mathcal{A}^{acc}$-generalized Drazin invertible (in short, ag-Drazin invertible) if there exists $x\in\{a\}^{'}$ such that 
\begin{equation*}
 xax=x \ {\rm and} \ a-a^{2}x\in \mathcal{A}^{acc}.
\end{equation*}
\end{definition}
Such $x$ is called a ag-Drazin inverse of $a$, and denote it by $a^{ad}$. We use $\mathcal{A}^{ad}$ to denote the set of all ag-Drazin invertible elements of $\mathcal{A}$. Obviously, $\mathcal{A}^{acc}\subseteq\mathcal{A}^{ad}$, and $\mathcal{A}^{d}=\mathcal{A}^{ed}\subseteq\mathcal{A}^{ad}$ since $\mathcal{A}^{qnil}\subset\mathcal{A}^{acc}$. Moreover, if $T\in\mathcal{A}$ is generalized Drazin-Riesz invertible then $T$ is ag-Drazin invertible when $\mathcal{A}= L(\mathcal{X})$. The following example indicates that if $a\in\mathcal{A}^{ad}$, then ag-Drazin inverses of $a$ may be not unique.
\begin{example}
Let $\mathcal{A}=L(\ell^{2}(\mathbb{N}))$ be the algebra of all bounded linear operators on $\ell^{2}(\mathbb{N})$. Given an operator $T: \ell^{2}\rightarrow\ell^{2}$ defined by 
$$T(x_{1}, x_{2}, x_{3}, \cdots)=(x_{1}, \frac{x_{2}}{2}, \frac{x_{3}}{3}, \cdots).$$ Set $$T_{1}(x_{1}, x_{2}, x_{3}, \cdots)=(x_{1},x_{2},0, 0, \cdots)$$ and $$T_{2}(x_{1}, x_{2}, x_{3}, \cdots)=(x_{1}, x_{2}, x_{3}, 0, \cdots).$$ Then one has 
$$TT_{1}(x_{1}, x_{2}, \cdots)=T(x_{1}, 2x_{2}, 0, \cdots)=(x_{1}, x_{2}, 0, \cdots),$$
$$T_{1}T(x_{1}, x_{2}, \cdots)=T_{1}(x_{1}, \frac{x_{2}}{2}, \frac{x_{3}}{3}, \cdots)=(x_{1}, x_{2}, 0, \cdots).$$ Thus, $TT_{1}=T_{1}T$. Moreover, since $$T_{1}TT_{1}(x_{1}, x_{2}, \cdots)=T_{1}(x_{1}, x_{2}, 0, \cdots)=(x_{1},x_{2},0, 0, \cdots)$$ and $$(T-T^{2}T_{1})(x_{1}, x_{2}, \cdots)=(0, 0, \frac{x_{3}}{3}, \frac{x_{4}}{4}, \cdots),$$ it follows that $T_{1}TT_{1}=T_{1}$ and ${\rm acc\,}\sigma(T-T^{2}T_{1})=\{0\}$. Similarly, one can  see that $TT_{2}=T_{2}T$, $T_{2}TT_{2}$ and ${\rm acc\,}\sigma(T-T^{2}T_{2})=\{0\}$. So both $T_{1}$ and $T_{2}$ are ag-Drazin inverse of $T$, but $T_{1}\neq T_{2}$.
\end{example}

In this paper, we investigate some basic properties of ag-Drazin inverses in Banach algebra $\mathcal{A}$. Using idempotent elements, we give some characterizations of ag-Drazin inverses, and prove that $a$ is ag-Drazin invertible if and only if there exists an idempotent $p\in\{a\}^{'}$ such that $a+p\in\mathcal{A}^{d}$ and $ap\in\mathcal{A}^{acc}$ if and only if  there exists an idempotent $f\in\{a\}^{'}$ such that $af$ is invertible in $f\mathcal{A}f$ and $a(1-f)\in((1-f)\mathcal{A}(1-f))^{acc}$. Moreover, we study  the ag-Drazin inverses by means of generalized Drazin spectrum, and show that $a\in\mathcal{A}$ is ag-Drazin invertible if and only if $0\notin{\rm acc\,}\sigma_{d}(a)$.

\section{The ag-Drazin inverses}\label{sec2}
In this section, we investigate the basic properties of ag-Drazin inverses in Banach algebra $\mathcal{A}$.  Harte defined the concept of a quasi-polar element in Banach algebra $\mathcal{A}$ (see \cite[Definition 7.5.2]{Har1988}). An element $a\in\mathcal{A}$ is said to be quasi-polar if there exists an idempotent $p\in\{a\}^{'}$ such that $a(1-p)\in\mathcal{A}^{qnil}$ and $p\in a\mathcal{A}\cap\mathcal{A}a$.  Koliha in his celebrated paper \cite{Kol1996} proved that $a$ is quasi-polar if and only if $a$ is generalized Drazin invertible.  In the following, we generalize this concept by means of new set $\mathcal{A}^{acc}$. 

\begin{definition}
An element $a\in\mathcal{A}$ is said to be acc-quasi-polar if there exists an idempotent $q\in\{a\}^{'}$ such that 
$$q\in a\mathcal{A}\cap\mathcal{A}a \ {\rm and} \ a(1-q)\in\mathcal{A}^{acc}.$$
\end{definition}

The following theorem proves that an element $a$ is acc-quasi-polar if and only if $a$ is ag-Drazin invertible.  Moreover, using idempotent elements, we also give some characterizations of ag-Drazin inverses. Let $a\in R$, and $p\in R$ be an idempotent element. By $\overline{p}$  we denote $1-p$. Then we can write
\begin{displaymath}
a=pap+pa\overline{p}+\overline{p}ap+\overline{p}a\overline{p}.
\end{displaymath}
Every idempotent element $p\in\mathcal{A}$ induces a representation of an element $a\in R$ given by the following matrix:
\begin{equation*}
 a=\left(\begin{array}{cccc}
 a_{1} & a_{2}\\
 a_{3} & a_{4}
 \end{array}\right)_{p},
\end{equation*}
where $a_{1}=pap, \, a_{2}=pa\overline{p}, \, a_{3}=\overline{p}ap$ and $a_{4}=\overline{p}a\overline{p}$.
We use $\sigma(a_{1}, p\mathcal{A}p)$ and $\sigma(a_{4}, \overline{p}\mathcal{A}\overline{p})$ to denote the specrum of $a_{1}$ in $p\mathcal{A}p$ and the specrum of $a_{2}$ in $\overline{p}\mathcal{A}\overline{p}$, respectively.
\begin{theorem} \label{AG}
Let $a\in\mathcal{A}$. The the following statements are equivalent.
\begin{enumerate}
\item[(1)] $a\in\mathcal{A}^{ad}$.
\item[(2)] $a$ is acc-quasi-polar.
\item[(3)] There exists an idempotent element $p\in\{a\}^{'}$ such that $a+p\in\mathcal{A}^{d}$ and $ap\in\mathcal{A}^{acc}$.
\item[(4)] There exists an idempotent element $f\in\{a\}^{'}$ such that   
\begin{equation*}
  a=\left(\begin{array}{cccc}
 a_{1} & 0\\
 0 & a_{2}
 \end{array}\right)_{f},
\end{equation*}
where $a_{1}$ is invertible in $f\mathcal{A}f$ and $a_{2}\in(\overline{f}\mathcal{A}\overline{f})^{acc}$.
\item[(5)] There exists an idempotent element $f\in\{a\}^{'}$ such that $a(1-f)\in\mathcal{A}^{acc}$ and $af+1-f$ is invertible in $\mathcal{A}$.
\end{enumerate}
\end{theorem}
\begin{proof}
$(1)\Rightarrow (2)$ Assume that $x$ is a ag-Drazin inverse of $a$. Let $e=xa$. Then $e^{2}=xaxa=xa=e$, $a(1-e)=a-a^{2}x\in\mathcal{A}^{acc}$ and $ea=xaa=axa=ae$. From $e=xa=ax$ one has $e\in a\mathcal{A}\cap\mathcal{A}a$. Thus, $a$ is acc-quasi-polar.

$(2)\Rightarrow(3)$ Suppose that $a$ is acc-quasi-polar. Then there exists $e^{2}=e\in\mathcal{A}$ such that $ea=ae$, $a(1-e)\in\mathcal{A}^{acc}$ and $e\in a\mathcal{A}\cap\mathcal{A}a$. Set $p=1-e$. Then $ap=pa$ and $ap=a(1-e)\in\mathcal{A}^{acc}$. From $e\in a\mathcal{A}\cap\mathcal{A}a$ that there exist $u,\ v\in\mathcal{A}$ such that $1-p=av=ua$, and hence 
\[
\begin{pmatrix}
u_{11}& u_{12} \\
u_{21} & u_{22}\\
\end{pmatrix}_{p}
\begin{pmatrix}
a_{1} & 0 \\
0& a_{2}\\
\end{pmatrix}_{p}=\begin{pmatrix}
a_{1} & 0 \\
0& a_{2}\\
\end{pmatrix}_{p}\begin{pmatrix}
v_{11}& v_{12} \\
v_{21} & v_{22}\\
\end{pmatrix}_{p}=\begin{pmatrix}
0 & 0 \\
0& 1-p\\
\end{pmatrix}_{p}
\]
which means that $u_{11}a_{1}=a_{1}v_{11}=0=u_{21}a_{1}=a_{1}v_{12}$ and $u_{22}a_{2}=a_{2}v_{22}=1-p$. It follows that $a_{2}$ is invertible in $\overline{p}\mathcal{A}\overline{p}$. Since 
\[
uav+p=\begin{pmatrix}
0 & 0 \\
0 & 1-p\\
\end{pmatrix}_{p}\begin{pmatrix}
v_{11}& v_{12} \\
v_{21} & v_{22}\\
\end{pmatrix}_{p}+\begin{pmatrix}
p & 0 \\
0& 0\\
\end{pmatrix}_{p}=\begin{pmatrix}
p & 0 \\
0 & v_{11}\\
\end{pmatrix}_{p},
\]
one has $uav+p$ is invertible in $\mathcal{A}$. Let $c(uav+p)=1$. Then $a(1-p)+p=c$, and hence $a+p=(1+ap)c=c(1+ap)$. Note that $ap\in\mathcal{A}^{acc}$, one has $1+ap$ is generalized Drazin invertible. By \cite[Theorem 5.2]{Kol1996}, $a+p$ is generalized Drazin invertible.

$(3)\Rightarrow(4)$  Since $a+p$ is generalized Drazin invertible, it follows that there exists an idempotent $h\in\{a\}^{''}$ such that $a+p+h$ is invertible and $(a+p)h\in\mathcal{A}^{qnil}$. From $ap=pa$ one has $ph=hp$. Let $$f=(1-h)(1-p).$$ Then $fa=af$ and $f^{2}=(1-h)(1-p)(1-h)(1-p)=(1-h)(1-p)=f$. Hence, one has 
\[
a=\begin{pmatrix}
a_{1} & 0 \\
0 & a_{2}\\
\end{pmatrix}_{f}
\]
where $a_{1}=af$ and $a_{2}=a\overline{f}$. It follows $(a+p+h)f=af$ that $a_{1}$ is invertible in $f\mathcal{A}f$. Next we prove that $a_{2}\in(\overline{f}\mathcal{A}\overline{f})^{acc}$.  It suffices to prove that for $\lambda\in\mathbb{C}\backslash\{0\}$, there exists $\epsilon>0$ such that $\mu\overline{f}-a_{2}$ is invertible in $\overline{f}\mathcal{A}\overline{f}$ if $\mu\in B(\lambda, \epsilon)^{\circ}$.
Obverse that $(a+p)h\in\mathcal{A}^{qnil}$, one has $ah(1-p)\in\mathcal{A}^{qnil}$ since $$ah(1-p)=(a+p)h(1-p)$$ and $$(a+p)h(1-p)=(1-p)(a+p)h.$$ For any $\lambda\in\mathbb{C}\backslash\{0\}$, $\lambda-ah(1-p)$ is invertible in $\mathcal{A}$, which means that there exists $\epsilon_{1}>0$ such that $v_{1}-ah(1-p)$ is invertible in $\mathcal{A}$ when $v_{1}\in B(\lambda, \epsilon_{1})^{\circ}$. Since $ap\in\mathcal{A}^{acc}$, it follows that $\lambda\in{\rm iso\,}\sigma(ap)$, and hence there exists $\epsilon_{2}>0$ such that $v_{2}-ap$ is invertible in $\mathcal{A}$ when $v_{2}\in B(\lambda, \epsilon_{2})^{\circ}$. Let $$\epsilon=\min\{\epsilon_{1}, \epsilon_{2}\}.$$ Then both $\mu-ah(1-p)$ and $\mu-ap$ are invertible in $\mathcal{A}$ when $\mu\in B(\lambda, \epsilon)^{\circ}$. Let $d_{1}$ be the inverse of $\mu-ah(1-p)$, and $d_{2}$ be the inverse of $\mu-ap$. Since $h, p, a$ are  commuting, the set $\{d_{1}, d_{2}, \overline{f}, h , p, a\}$ is commutative. 
Moreover, as $$a_{2}=a\overline{f}=ah(1-p)+ap(1-h)+aph,$$ one has $$\mu\overline{f}-a_{2}=\mu h(1-p)+\mu p(1-h)+\mu ph-ah(1-p)-ap(1-h)-aph.$$ Define $$z=\overline{f}\big(d_{1}h(1-p)+d_{2}p(1-h)+d_{2}ph\big)\in\overline{f}\mathcal{A}\overline{f}.$$  Then 
\begin{equation*}
\begin{split}
z(\mu\overline{f}-a_{2})&=\overline{f}\big(d_{1}h(1-p)+d_{2}p(1-h)+d_{2}ph\big)\big(\mu h(1-p)+\mu p(1-h)\\
&\quad +\mu ph-ah(1-p)-ap(1-h)-aph\big)\\
&=\overline{f}\Big(d_{1}h(1-p)\big((\mu h(1-p))-ah(1-p)\big)+d_{2}p(1-h)\big(\mu p(1-h)\\
&\quad -ap(1-h)\big)+ d_{2}ph\big(\mu ph-aph\big)\Big)\\
&=\overline{f}\Big(d_{1}(\mu-ah(1-p))h(1-p)+ d_{2}(\mu-ap)p(1-h)+d_{2}(\mu-ap)ph\Big)\\
&=\overline{f}(h(1-p)+p(1-h)+ph)\\
&=\overline{f}(h+p-ph)=\overline{f}.
\end{split}
\end{equation*}
Thus, $\mu\overline{f}-a_{2}$ is invertible in $\overline{f}\mathcal{A}\overline{f}$, which implies that $a_{2}\in(\overline{f}\mathcal{A}\overline{f})^{acc}$.

$(4)\Rightarrow(5)$ Let $c\in f\mathcal{A}f$ satisfy $ca_{1}=a_{1}c=f$, and set $y=c+1-f$. Then $$y(af+1-f)=caf+1-f=1.$$ Similarly, one has $(af+1-f)c=1$. Thus, $af+1-f$ is invertible in $\mathcal{A}$. Moreover, since $a(1-f)=a_{2}\in(\overline{f}\mathcal{A}\overline{f})^{acc}$ and $\sigma(a(1-f))=\{0\}\cup\sigma(a_{2}, \overline{f}\mathcal{A}\overline{f})$, one can easily obtain that $a(1-f)\in\mathcal{A}^{acc}$.

$(5)\Rightarrow(1)$ Let $g$ be the inverse of $af+1-f$. Observe that $(af+1-f)f=af$ means that $f=gaf$. Set $x=gf$. Then $xa=gfa=f=ax$ and $xax=gfagf=fgf=gf=x$, and moreover, one has $a-a^{2}x=a-aagf=a-af=a(1-f)\in\mathcal{A}^{acc}$. Thus, $a$ is ag-Drazin invertible.
\end{proof}

Suppose that $a\in\mathcal{A}^{ad}$. Let ${\rm card}(a)$ denote the number of elements in the $\sigma(a-a^{2}a^{ad})$. If there exists $a^{ad}$ such that $\sigma(a-a^{2}a^{ad})$ is finite set, then ${\rm card}(a)<+\infty$. Otherwise, ${\rm card}(a)=+\infty$.  Since the spectrum of element is not empty, it follows  that ${\rm card}(a)>0$. 
\begin{proposition} \label{AG2}
Let $a\in\mathcal{A}$. Then the following statements are equivalent.
\begin{enumerate}
\item[(1)] $a\in\mathcal{A}^{ad}$ with ${\rm card}(a)<+\infty$.
\item[(2)] $a\in\mathcal{A}^{ed}$.
\item[(3)] $a\in\mathcal{A}^{d}$.
\end{enumerate}
\end{proposition}
\begin{proof}
$(1)\Rightarrow (2)$ Assume that $a$ is ag-Drazin invertible and ${\rm card}(a)<+\infty$. Then there exists $x\in\mathcal{A}$ as an ag-Drazin inverse of $a$ such that $\sigma(a-a^{2}x)$ is finite set, and hence $0\notin{\rm acc\,}\sigma(a-a^{2}x)$, which means that $a-a^{2}x\in\mathcal{A}^{d}$. Therefore, $a$  is eg-Drazin invertible.

$(2)\Rightarrow (3)$ see \cite[Theorem 2]{Mos2020}.

$(3)\Rightarrow (1)$ It is obvious since $\mathcal{A}^{qnil}\subseteq\mathcal{A}^{acc}$.
\end{proof}

\begin{remark}
From the above proposition, one can see that if $a\in\mathcal{A}^{ad}$ with ${\rm card}(a)=+\infty$ then $a$ is not generalized Drazin invertible.
\end{remark}

\begin{proposition}\label{AG1}
Let $a\in\mathcal{A}^{ad}$ with ${\rm card}(a)=+\infty$. Then 
\begin{enumerate}
\item[(1)] There exists an idempotent element $p\in\{a\}$ such that $ap$ is invertible in $p\mathcal{A}p$ and $a(1-p)\in(\overline{p}\mathcal{A}\overline{p})^{acc}$ with infinite spectrum.
\item[(2)] There exists a sequence of nonzero isolated spectrum point of $a$ which converges to 0.
\end{enumerate}
\end{proposition}
\begin{proof}
Suppose that $a\in\mathcal{A}^{ad}$ with ${\rm card}(a)=+\infty$. By Theorem \ref{AG}, there exists an idempotent element $p\in\{a\}^{'}$ such that 
\begin{equation*}
 a=\left(\begin{array}{cccc}
 a_{1} & 0\\
 0 & a_{2}
 \end{array}\right)_{p},
\end{equation*}
where $a_{1}$ is invertible in $p\mathcal{A}p$ and $a_{2}\in(\overline{p}\mathcal{A}\overline{p})^{acc}$. Then $B(0; \epsilon)\subseteq\mathbb{C}\backslash\sigma(a_{1}, p\mathcal{A}p)$ when $\epsilon$ is enough small. Moreover, since ${\rm card}(a)=+\infty$, $0\in{\rm acc\,}\sigma(a)$ by Proposition \ref{AG2}, and hence there exists a sequence $\{\lambda_{n}\}\subseteq B(0; \epsilon)^{\circ}\cap\sigma(a)$ such that $\lambda_{n}\rightarrow 0 \ (n\rightarrow +\infty)$.  Using the fact $\sigma(a)=\sigma(a_{1}, p\mathcal{A}p)\cup\sigma(a_{2}, \overline{p}\mathcal{A}\overline{p})$, one can see that $$B(0; \epsilon)^{\circ}\cap\sigma(a)=B(0; \epsilon)^{\circ}\cap\sigma(a_{2}, \overline{p}\mathcal{A}\overline{p}),$$ and moreover, it follows from $a_{2}\in (\overline{p}\mathcal{A}\overline{p})^{acc}$ that $\{\lambda_{n}\}\subseteq{\rm iso\,}\sigma(a_{2}, \overline{p}\mathcal{A}\overline{p})$, which means that $\{\lambda_{n}\}\subseteq{\rm iso\,}\sigma(a)$. This proves (1) and (2).
\end{proof}
\begin{remark} \label{RAG}
Let $a\in\mathcal{A}^{ad}$ with ${\rm  card}(a)=+\infty$. By Proposition \ref{AG1}, one can see that there exists an idempotent $p\in\{a\}^{'}$ such that 
\begin{equation*}
 a=\left(\begin{array}{cccc}
 a_{1} & 0\\
 0 & a_{2}
 \end{array}\right)_{p},
\end{equation*}
where $a_{1}$ is invertible in $p\mathcal{A}p$, $\sigma(a_{2}, \overline{p}\mathcal{A}\overline{p})=\{0, \lambda_{1}, \lambda_{2}, \lambda_{3}, \cdots\}$ $(\lvert\lambda_{1}\rvert\geq\lvert\lambda_{2}\rvert\geq\lvert\lambda_{3}\rvert\geq\cdots$ , and $\lambda_{n}\rightarrow 0$ when $n\rightarrow +\infty$).  Thus, one can find $\epsilon>0$ such that $B(0, \epsilon)\subseteq\mathbb{C}\backslash\sigma(a_{1}, p\mathcal{A}p)$, and $\{\lambda_{n}\}_{n\geq N}\subset B(0, \epsilon)$ for some $N\in\mathbb{N}$. This also implies that $\{\lambda_{n}\}_{n\geq N}$ is a sequence of nonzero isolated spectrum points of $a$. Moreover, one can easily check that $$\sigma(a_{1}, p\mathcal{A}p)\cap\sigma(a, \overline{p}\mathcal{A}\overline{p})\subseteq\{\lambda_{1}, \lambda_{2}, \cdots, \lambda_{N}\}.$$ Similar to \cite{Aba2021}, for each $n\geq N$, one can define two closed set: $$\Delta_{n}=\{0, \lambda_{n+1}, \lambda_{n+2}, \cdots\}$$ and $$\Delta_{n}^{'}=(\sigma(a_{1}, p\mathcal{A}p)\backslash(\sigma(a_{1}, p\mathcal{A}p)\cap\sigma(a, \overline{p}\mathcal{A}\overline{p})))\cup\{\lambda_{1}, \lambda_{2}, \cdots, \lambda_{n}\}=\sigma(a)\backslash\Delta_{n}.$$ Obviously, for all $n\geq N$, $\Delta_{n}$ and $\Delta_{n}^{'}$ are disjoint  nonempty closed subset of $\sigma(a)$. 
\end{remark}

Now, we state a lemma which will be need in the following.
\begin{lemma} {\rm \cite[Page 204-205]{Con1990}} \label{CON}
Let $a\in\mathcal{A}$. If $\sigma(a)= F_{1}\cup F_{2}$, where $F_{1}$ and $F_{2}$  are disjoint  nonempty closed set, then there exists a nontrival idempotent element $e\in\{a\}^{''}$ such that 
\begin{equation*}
 a=\left(\begin{array}{cccc}
 a_{1} & 0\\
 0 & a_{2}
 \end{array}\right)_{e},
\end{equation*}
where $\sigma(a_{1})=F_{1}\cup\{0\}$, $\sigma(a_{2})=F_{2}\cup\{0\}$, $\sigma(a_{1}, e\mathcal{A}e)=F_{1}$ and $\sigma(a_{2}, \overline{e}\mathcal{A}\overline{e})=F_{2}$.
\end{lemma}

 It was proved in \cite[Theorem 4.4]{Kol1996} that if $a$ is generalized Drazin invertible, then there exists an idempotent element $p\in\{a\}^{'}$ such taht $a^{d}=(a+p)^{-1}(1-p)$. For ag-Drazin inverses, one has
\begin{theorem} \label{DCAD}
Let $a\in\mathcal{A}$. Then the following statements are equivalent.
\begin{enumerate}
\item[(1)] $a\in\mathcal{A}^{ad}$.
\item[(2)] There exists an idempotent element $p\in\{a\}^{''}$ such that $a+p$ is invertible in $\mathcal{A}$ and $ap\in\mathcal{A}^{acc}$. In this  case, $b=(a+p)^{-1}(1-p)$ is a ag-Drazin inverse of $a$.
\item[(3)] There exists an idempotent element $p\in\{a\}^{'}$ such that $a+p$ is invertible in $\mathcal{A}$ and $ap\in\mathcal{A}^{acc}$.
\end{enumerate}
\end{theorem}
\begin{proof}
$(1)\Rightarrow (2)$ Assume that $a\in\mathcal{A}^{ad}$. If ${\rm card}(a)<+\infty$, then $a\in\mathcal{A}^{d}$ by Proposition \ref{AG2}, and hence (2) holds. On the other hand, if ${\rm card}(a)=+\infty$, then let $n\geq N$ be enough large such that $\sup_{\lambda\in\Delta_{n}}\lvert\lambda\rvert<\frac{1}{2}$ and $\sigma(a)=\Delta_{n}^{'}\cup\Delta_{n}$ due to Remark \ref{RAG}. It follows from Lemma \ref{CON} that there exists an idempotent element $e\in\{a\}^{''}$ such that 
\begin{equation*}
 a=\left(\begin{array}{cccc}
 a_{1} & 0\\
 0 & a_{2}
 \end{array}\right)_{e},
\end{equation*}
where $\sigma(a_{1})=\Delta_{n}^{'}\cup\{0\}$, $\sigma(a_{2})=\Delta_{n}$, $\sigma(a_{1}, e\mathcal{A}e)=\Delta_{n}^{'}$ and $\sigma(a_{2}, \overline{e}\mathcal{A}\overline{e})=\Delta_{n}$. Thus, $\sigma(a_{2})=\{0, \lambda_{n+1}, \lambda_{n+2}, \cdots\}$, which means that $a(1-e)\in\mathcal{A}^{acc}$.
Moreover, since $0\notin\sigma(a_{1}, e\mathcal{A}e)$, one has $a_{1}$ is invertible in $e\mathcal{A}e$. Notice that $\sup_{\lambda\in\Delta_{n}}\lvert\lambda\rvert<\frac{1}{2}$ implies 
\begin{equation*}
 a+1-e=\left(\begin{array}{cccc}
 a_{1} & 0\\
 0 & 1-e+a_{2}
 \end{array}\right)_{e} \ is\ invertible \ in \ \mathcal{A}.
\end{equation*}
Let $p=1-e$. Obviously, $a+p$ is invertible in $\mathcal{A}$ and $ap\in\mathcal{A}^{acc}$. Set $b=(a+p)^{-1}(1-p)$. Then $$bab=(a+p)^{-1}(1-p)(a+p)(a+p)^{-1}(1-p)=(a+p)^{-1}(1-p)=b$$ and $$a-a^{2}b=a-a(a+p)(a+p)^{-1}(1-p)=ap\in\mathcal{A}^{acc}.$$ Thus, $b$ is a ag-Drazin inverse of $a$.

$(2)\Rightarrow(3)$ It is obvious.

$(3)\Rightarrow(1)$ Since $a+p$ is invertible, one has $a+p\in\mathcal{A}^{d}$. By Theorem \ref{AG}, $a$ is ag-Drazin invertible.
\end{proof}
\begin{corollary}
If $a\in\mathcal{A}^{ad}$ with ${\rm Card}(a)=+\infty$, then ag-Drazin inverses of $a$ are not unique. 
\end{corollary}
\begin{proof}
 Let $n\geq N$ be enough large such that $\sup_{\lambda\in\Delta_{n}}\lvert\lambda\rvert<\frac{1}{2}$ due to Remark \ref{RAG}. Then $$\Delta_{n}=\{0, \lambda_{n+1}, \lambda_{n+2}, \cdots\}.$$ By Theorem \ref{DCAD}, there exists an idempotent element $e\in\{a\}^{''}$ such that 
\begin{equation*}
 a=\left(\begin{array}{cccc}
 a_{1} & 0\\
 0 & a_{2}
 \end{array}\right)_{e},
\end{equation*}
where $a_{1}$ is invertible in $e\mathcal{A}e$ and $a_{2}\in(\overline{e}\mathcal{A}\overline{e})^{acc}$. In this case, 
\begin{equation*}
 b=\left(\begin{array}{cccc}
 b_{1} & 0\\
 0 & 0
 \end{array}\right)_{e}
\end{equation*}
is a ag-Drazin inverse of $a$, where $b_{1}$ is inverse of $a_{1}$ in $e\mathcal{A}e$. Let $n_{1}>n$ and $$\Delta_{n_{1}}=\{0, \lambda_{n_{1}+1}, \lambda_{n_{1}+2}, \cdots\}.$$ Then  there exists an idempotent element $f\in\{a\}^{''}$ such that 
\begin{equation*}
 a=\left(\begin{array}{cccc}
 a^{'}_{1} & 0\\
 0 & a^{'}_{2}
 \end{array}\right)_{f},
\end{equation*}
where $a^{'}_{1}$ is invertible in $f\mathcal{A}f$ and $a^{'}_{2}\in(\overline{f}\mathcal{A}\overline{f})^{acc}$. In this case, 
\begin{equation*}
 c=\left(\begin{array}{cccc}
 c_{1} & 0\\
 0 & 0
 \end{array}\right)_{f}
\end{equation*}
is a ag-Drazin inverse of $a$, where $c_{1}$ is a inverse of $a^{'}_{1}$ in $f\mathcal{A}f$. It follows from $\sigma(a_{1})=\Delta^{'}_{n}\cup\{0\}\neq\Delta^{'}_{n_{1}}\cup\{0\}=\sigma(a^{'}_{1})$ that $a_{1}\neq a^{'}_{1}$, that is, $e\neq f$. Next, we will prove $b\neq c$. Indeed, if $c=b$, then $c_{1}=b_{1}\in e\mathcal{A}e\cap f\mathcal{A}f$, and hence $$e=b_{1}a_{1}=b_{1}ea=b_{1}fa=c_{1}a^{'}_{1}=f.$$ It is a contradiction with $e\neq f$. Thus,  ag-Drazin inverses of $a$ are not unique.
\end{proof}
\begin{remark}
From the above corollary, we know that if $a\in \mathcal{A}^{ad}\backslash\mathcal{A}^{d}$, then the set of all ag-Drazin inverse of $a$ is an infinite set. Therefore, if $a\in \mathcal{A}^{ad}$, then $a$ has unique ag-Drazin inverse if and only if $a$ is generalized Drazin invertible.
\end{remark}
Let $\overline{M}$ denote the closure of the set $M$.
The proof of the following theorem about ag-Drazin inverses is inspired by
that of Theorem 2.9 in \cite{Aba2021}. 
\begin{theorem} \label{GDS}
Let $a\in\mathcal{A}$. Then $a\in\mathcal{A}^{ad}$ if and only if $0\notin{\rm acc}\sigma_{d}(a)$.
\end{theorem}
\begin{proof}
Suppose that $a\in\mathcal{A}^{ad}$. By Theorem \ref{AG}, there exists an idempotent $f\in\{a\}^{'}$ such that   
\begin{equation*}
  a=\left(\begin{array}{cccc}
 a_{1} & 0\\
 0 & a_{2}
 \end{array}\right)_{f},
\end{equation*}
where $a_{1}$ is invertible in $f\mathcal{A}f$ and $a_{2}\in(\overline{f}\mathcal{A}\overline{f})^{acc}$. Thus, there exists $\epsilon>0$ such that $a_{1}-\lambda f$ is invertible and $\lambda\in{\rm iso}\sigma(a_{2}, \overline{f}\mathcal{A}\overline{f})$ if $\lambda\in B(0; \epsilon)^{\circ}$, which means that $\lambda\notin{\rm acc}\sigma_{d}(a)$.

Conversely, assume that $0\notin{\rm acc}\sigma_{d}(a)$. If $0\notin{\rm acc}\sigma(a)$, then $a\in\mathcal{A}^{d}$, and hence $a\in\mathcal{A}^{ad}$. On the other hand, if $0\in{\rm acc}\sigma(a)$, then $0\in{\rm acc}({\rm iso}\sigma(a))$. We distinguish two cases.\\
\underline{Case 1}. If ${\rm acc}({\rm iso}\sigma(a))\neq\{0\}$. Set $\delta=\inf\limits_{\lambda\in{\rm acc}({\rm iso}\sigma(a))\backslash\{0\}}\lvert\lambda\rvert$. Then there exists a sequence $(\lambda_{n})\subset{\rm acc}({\rm iso}\sigma(a))\backslash\{0\}$ such that $\lambda_{n}\longrightarrow \lambda \, (n\longrightarrow\infty)$. Since $(\lambda_{n})\subset{\rm acc}({\rm iso}\sigma(a))\backslash\{0\}\subset{\rm acc}\sigma(a)=\sigma_{d}(a)$ and $0\notin{\rm acc}\sigma_{d}(a)$, one has $\delta>0$. Obviously, the set $F_{1}=B(0, \frac{\delta}{2})\cap\overline{{\rm iso}\sigma(a)}$ is a closed subset of $\sigma(a)$. Let $F_{2}=(\sigma_{d}(a)\backslash\{0\})\cup(\overline{{\rm iso}\sigma(a)}\backslash F_{1})$.  Then $F_{1}\cap F_{2}\neq\emptyset$ and $\sigma(a)=F_{1}\cup F_{2}$. Since $0\notin {\rm acc}\sigma_{d}(a)$, one can obtain that $\sigma_{d}(a)\backslash\{0\}$ is closed. Moreover, without loss of generality, we may assume that $\mu\in\overline{\overline{{\rm iso}\sigma(a)}\backslash F_{1}}\subset\sigma(a)$. There exists a sequence $(\mu_{n})\subset\overline{{\rm iso}\sigma(a)}\backslash F_{1}$ such that $\mu_{n}\longrightarrow \mu \, (n\longrightarrow\infty)$ , and hece $\mu\neq 0$. If $\mu\in F_{1}$, then $\mu\in{\rm iso}\sigma(a)$. It is a contradiction. Thus, $\mu\in F_{2}$, which means that $F_{2}$ is closed.  From Lemma \ref{CON}, there exists an idempotent element $e\in\{a\}^{''}$ such that 
\begin{equation*}
 a=\left(\begin{array}{cccc}
 a_{1} & 0\\
 0 & a_{2}
 \end{array}\right)_{e},
\end{equation*}
where $\sigma(a_{1}, e\mathcal{A}e)=F_{1}$ and $\sigma(a_{2}, \overline{e}\mathcal{A}\overline{e})=F_{2}$. Hence, $a_{1}\in  (e\mathcal{A}e)^{acc}$ and $a_{2}$ is invertible in $\overline{e}\mathcal{A}\overline{e}$. By Theorem \ref{AG}, $a$ is ag-Drazin invertible.\\
\underline{Case 2} If ${\rm acc}({\rm iso}\sigma(a))=\{0\}$, then there exists $\delta>0$ such that $\sup\limits_{\lambda\in{\rm iso}\sigma(a)}\lvert\lambda\rvert$. Let $F_{1}=B(0, \frac{\delta}{2})\cap\overline{{\rm iso}\sigma(a)}$ and $F_{2}=(\sigma_{d}(a)\backslash\{0\})\cup(\overline{{\rm iso}\sigma(a)}\backslash F_{1})$. Similar the proof of Case 1, one can obtain that $a$ is ag-Drazin invertible.
\end{proof}
Let $\sigma_{ad}(a)$ denote the ag-Drazin spectrum of $a$. From Theorem \ref{GDS}, one has $\sigma_{ad}(a)={\rm acc}\sigma_{d}(a)$. It was proved in \cite{YKZ} that $\sigma_{d}(ac)=\sigma_{d}(bd)$ provided that $a, b, c, d\in\mathcal{A}$ satisfy $acd=dbd$ and $dba=aca$. In this case, one has $\sigma_{ad}(ac)={\rm acc}\sigma_{d}(ac)={\rm acc}\sigma_{d}(bd)=\sigma_{ad}(bd)$. Thus, one can obtain the following result.
\begin{corollary}
Let $a, b, c ,d \in\mathcal{A}$. If $a, b, c, d$ satisfy $acd=dbd$ and $dba=aca$, then $\sigma_{ad}(ac)=\sigma_{ad}(bd)$.
\end{corollary}
From the above corollary, one can easily get the following two results.
\begin{corollary}
Let $a, b, c \in\mathcal{A}$. If $a, b, c$ satisfy $aca=aba$, then $\sigma_{ad}(ac)=\sigma_{ad}(ba)$.
\end{corollary}
\begin{corollary}
If $a, b \in\mathcal{A}$, then $\sigma_{ad}(ab)=\sigma_{ad}(ba)$.
\end{corollary}

In \cite{Kol1996}, Koliha proved that if $a\in\mathcal{A}^{d}$ and $b\in\mathcal{A}^{qnil}\cap\{a\}^{'}$ then $a+b\in\mathcal{A}^{d}$. For ag-Drazin inverses, one has
\begin{proposition}
If $a\in\mathcal{A}^{ad}$, and let $b\in\mathcal{A}^{qnil}\cap\{a\}^{'}$, then $a+b\in\mathcal{A}^{ad}$.
\end{proposition}
\begin{proof}
By Theorem \ref{DCAD}, there exists an idempotent $p\in\{a\}^{''}$ such that $a+p$ is invertible in $\mathcal{A}$ and $ap\in\mathcal{A}^{acc}$. Since $ab=ba$, the set $\{a, b, p\}$ is commutative. Then $a+b+p\in\mathcal{A}^{d}$. Moreover, since $ab=ba$ and $bp\in\mathcal{A}^{qnil}$, it follows from Gelfand theory
for commutative Banach algebras that $$\sigma((a+b)p)\subset\sigma(ap)+\sigma(bp)=\sigma(ap),$$ which means that $(a+b)p\in\mathcal{A}^{acc}$. In view of Theorem \ref{AG}, $a+b$ is ag-Drazin invertible.
\end{proof}

\begin{proposition}
Let $a\in\mathcal{A}^{ad}$ and $b\in\mathcal{A}^{d}$. If $ab=ba=0$, then $a+b\in\mathcal{A}^{ad}$.
\end{proposition}
\begin{proof}
By Theorem \ref{DCAD}, there exists $x\in\{a\}^{''}$ such that  $x$ is a ag-Drazin inverse of $a$. Then $xb=bx=x^{2}ab=0$ and $ab^{d}=b^{d}a=0$, and hence $$(a+b)(x+b^{d})^{2}=ax^{2}+b(b^{d})^{2}=x+b^{d}$$ and $a+b-(a+b)^{2}(x+b^{d})=(a-a^{2}x)+(b-b^{2}b^{d})$. Since $b-b^{2}b^{d}\in\mathcal{A}^{qnil}$, one has  $a+b-(a+b)^{2}(x+b^{d})\in\mathcal{A}^{acc}$. Thus, $a+b$ is ag-Drazin invertible.
\end{proof}
The following decomposition of an ag-Drazin invertible element generalizes the
important core-quasinilpotent decomposition of a generalized Drazin invertible element given
in \cite[Theorem 6.4]{Kol1996}.
\begin{theorem}
Let $a\in\mathcal{A}$. Then $a\in\mathcal{A}^{ad}$ if and only if $a=x+y$, where $xy=0=yx$, $x\in\mathcal{A}^{\sharp}$ and $y\in\mathcal{A}^{acc}$.
\end{theorem}
\begin{proof}
Supposet that $a\in\mathcal{A}^{ad}$. Let $p=ab$, where $b$ is a ag-Drazin inverse. Then $p^{2}=p$ and $a(1-p)\in\mathcal{A}^{acc}$. Since $apbap=pap=ap$ and $bapb=b$, $ap\in\mathcal{A}^{\sharp}$. Let $x=ap$ and $y=a(1-p)$. Then $a=x+y$, $xy=0=yx$, $x\in\mathcal{A}^{\sharp}$ and $y\in\mathcal{A}^{acc}$.

Conversely, let $b=x^{\sharp}$. It follows from $by=b^{2}xy=0=yb$ that $ba=b(x+y)=bx=xb=(x+y)b=ab$, $bab=bxb=b$ and $a-aba=(x+y)-(x+y)b(x+y)=x+y-xbx=y\in\mathcal{A}^{acc}$. Thus, $a\in\mathcal{A}^{acc}$. 
\end{proof}

\section{ag-Drazin inverse for operators}
Let $L(\mathcal{X})$ denote the algebra of all bounded linear operators on a complex Banach space $\mathcal{X}$. In this section,
we characterize an ag-Drazin inverse of $T\in L(\mathcal{X})$ in terms of the direct sum of operators. 
Let $\mathcal{X}_{1}$ and $\mathcal{X}_{2}$ be two subspaces of $\mathcal{X}$. A closed subspace $\mathcal{X}_{1}$ is complement if there exists closed subspace $\mathcal{X}_{2}$ such that $\mathcal{X}=\mathcal{X}_{1}\oplus\mathcal{X}_{2}$. For $T\in B(\mathcal{X})$, we call $\mathcal{X}_{1}$ is $T$-invariant if $T(\mathcal{X}_{1})\subseteq\mathcal{X}_{1}$. We define $T\vert_{\mathcal{X}_{1}}: \ \mathcal{X}_{1}\rightarrow\mathcal{X}_{1}$ by $T\vert_{\mathcal{X}_{1}}x=Tx,\ x\in\mathcal{X}_{1}$. If $\mathcal{X}_{1}$ and $\mathcal{X}_{2}$ are two closed $T$-invariant subspaces of $\mathcal{X}$ such that $\mathcal{X}=\mathcal{X}_{1}\oplus\mathcal{X}_{2}$, it is said that $T$ is completely reduced by the pair $(\mathcal{X}_{1}, \mathcal{X}_{2})$, and denote it by $(\mathcal{X}_{1}, \mathcal{X}_{2})\in\hbox{Red}(T)$.

\begin{proposition}
Let $T\in L(\mathcal{X})$. Then the following statements are equivalent.
\begin{enumerate}
\item[(1)] $a\in\mathcal{A}^{ad}$.
\item[(2)] There exists $(\mathcal{X}_{1}, \mathcal{X}_{2})\in{\rm Red}(T)$ such that $T=T\vert_{\mathcal{X}_{1}}\oplus T\vert_{\mathcal{X}_{2}}$, where $T\vert_{\mathcal{X}_{1}}$ is generalized Drazin invertible and $\sigma_{d}(T\vert_{\mathcal{X}_{2}})\subseteq\{0\}$.
\item[(3)] There exists $(\mathcal{M}, \mathcal{N})\in{\rm Red}(T)$ such that $T=T\vert_{\mathcal{M}}\oplus T\vert_{\mathcal{N}}$, where $T\vert_{\mathcal{M}}$ is invertible and $\sigma_{d}(T\vert_{\mathcal{N}})\subseteq\{0\}$.
\end{enumerate}
\end{proposition}
\begin{proof}
$(1)\Longrightarrow (2)$ By Theorem \ref{AG}, there exists an idempotent element $p\in\{a\}^{'}$ such that $a+p$ is generalized Drazin invertible and $\sigma_{d}(ap)\subseteq\{0\}$. For $\mathcal{X}_{1}=N(P)$ and $\mathcal{X}_{2}=R(P)$, we have $(\mathcal{X}_{1}, \mathcal{X}_{2})\in{\rm Red}(T)$ and $T\vert_{\mathcal{X}_{1}}=(T+P)\vert_{\mathcal{X}_{1}}$ is generalized Drazin invertible. Since $T\vert_{\mathcal{X}_{2}}=(TP)\vert_{\mathcal{X}_{2}}$, $\sigma_{d}(T\vert_{\mathcal{X}_{2}})\subseteq\{0\}$.

$(2)\Longrightarrow (3)$ Sicne $T\vert_{\mathcal{X}_{1}}$ is generalized Drazin invertible, there exists $(\mathcal{X}_{11}, \mathcal{X}_{12})\in{\rm Red}(T\vert_{\mathcal{X}_{1}})$ such that $T\vert_{\mathcal{X}_{11}}$ is invertible and $T\vert_{\mathcal{X}_{12}}$ is quasinilpotent. Let $\mathcal{M}=\mathcal{X}_{11}$ and $\mathcal{N}=\mathcal{X}_{12}\oplus\mathcal{X}_{2}$. Then $(\mathcal{M}, \mathcal{N})\in{\rm Red}(T)$ such that $T=T\vert_{\mathcal{M}}\oplus T\vert_{\mathcal{N}}$, where $T\vert_{\mathcal{M}}$ is invertible and $\sigma_{d}(T\vert_{\mathcal{N}})\subseteq\{0\}$.

$(3)\Longrightarrow (1)$ Let $S=S_{1}\oplus 0\vert_{\mathcal{N}}$, where $S_{1}$ is inverse of $T\vert_{\mathcal{M}}$. Obviously, $S$ is a ag-Drazin inverse of $T$.
\end{proof}
Recall that $F\in L(\mathcal{X})$ is a power finite rank operator if there exists some integer $n$ such that $F^{n}$ is a finite rank operator. The following indicates that $L(\mathcal{X})^{ad}$ is invariant under commuting power finite rank perturbation.
\begin{proposition} \label{AGF}
Let $T\in L(\mathcal{X})^{ad}$, and let $F\in L(\mathcal{X})$ is a power finite rank operator. If $TF=FT$, then $T+F\in L(\mathcal{X})^{ad}$. 
\end{proposition}
\begin{proof}
In view of Theorem \ref{GDS}, one has $0\notin{\rm acc}\sigma_{d}(T)$.
It follows from \cite[Lemma 2.1]{YK2016} that $\sigma_{d}(T)=\sigma_{d}(T+F)$, and hence $0\notin{\rm acc}\sigma_{d}(T+F)$. Thus, $T+F$ is ag-Drazin invertible.  
\end{proof}
From Proposition \ref{AGF}, one can easily get the following result.
\begin{corollary} 
Let $T\in L(\mathcal{X})^{ad}$, and let $F\in L(\mathcal{X})$ is a finite rank operator. If $TF=FT$, then $T+F\in L(\mathcal{X})^{ad}$. 
\end{corollary}

\bmhead{Acknowledgments}

The authors would like to thank the anonymous referees for helpful comments and suggestions. This research is supported by the National Natural Science Foundation of China (Grant number 11871303).


\bibliography{sn-bibliography}


%

\end{document}